\documentclass[12pt]{amsart}

\usepackage{amssymb}%
\usepackage{amsmath}%
\usepackage{mathrsfs}
\usepackage{amsthm}%

\usepackage{tikz}%

\usepackage{fullpage}%
\usepackage{xcolor}%

\allowdisplaybreaks%

{\theoremstyle{plain}%
 \newtheorem{theorem}{Theorem}

}
{\theoremstyle{remark}

}
{\theoremstyle{definition}
\newtheorem{definition}{Definition}
\newtheorem{example}{Example}
}

\begin{document}

\begin{center}
 {\large A lift of West's stack-sorting map to partition diagrams}

 \ 

 {\textsc{John M. Campbell} }

 {\footnotesize Department of Mathematics, Toronto Metropolitan University, 350 Victoria St, Toronto, ON M5B 2K3}

 {\footnotesize {\tt jmaxwellcampbell@gmail.com}}

\end{center}

 \ 

\begin{quote}
 {\small {\bf Abstract.} We introduce a lifting of West's stack-sorting map $s$ to partition diagrams, which are combinatorial objects indexing bases of 
 partition algebras. Our lifting $\mathscr{S}$ of $s$ is such that $\mathscr{S}$ behaves in the same way as $s$ when restricted to diagram basis elements 
 in the order-$n$ symmetric group algebra as a diagram subalgebra of the partition algebra $\mathscr{P}_{n}^{\xi}$. We then introduce a lifting of the 
 notion of $1$-stack-sortability, using our lifting of $s$. By direct analogy with Knuth's famous result that a permutation is $1$-stack-sortable if and only if it 
 avoids the pattern 231, we prove a related pattern-avoidance property for partition diagrams, as opposed to permutations, according to what we 
 refer to as \emph{stretch-stack-sortability}. } 
\end{quote}

 \ 



 \ 

\section{Introduction}
 For a permutation $p$ in the symmetric group $S_{n}$, we write 
\begin{equation}\label{pLnR}
 p = L n R, 
\end{equation} 
 letting $p$ be denoted as 
 a string or tuple given by the entries of the bottom row of the two-line notation for $p$. 
 We then let \emph{West's stack-sorting map} $s$ 
 \cite{Defant2020Comb,Defant2021Discrete,Defant2021European,Defant2021Enumer,Defant2020Electron,Defant2022Theory,Defant2020Lothar,DefantKravitz2020,DefantZheng2021}
 (cf.\ \cite{West1990}) 
 be defined recursively so that 
\begin{equation}\label{sLsRn}
 s(p) = s(L) s(R) n
\end{equation}
 and so that $s$ maps permutations to permutations and sends the empty permutation to itself. Our notation and terminology concerning this mapping 
 are mainly based on references such as 
 \cite{Defant2020Comb,Defant2021Discrete,Defant2021European,Defant2021Enumer,Defant2020Electron,Defant2022Theory,Defant2020Lothar,DefantKravitz2020,DefantZheng2021}. 
 In this article, we introduce a lifting of $s$ so as to allow combinatorial objects known as \emph{partition diagrams} as input. 

 The problem of generalizing West's stack-sorting map has been considered in a number of different contexts. Notably, the stack-sorting map $s$ has been 
 generalized to Coxeter groups \cite{Defant2022Theory} and to words \cite{DefantKravitz2020}, and Cerbai, Claesson, and Ferrari generalized 
\begin{equation}\label{smap}
 s\colon S_{n} \to S_{n}
\end{equation} 
 to a function $s_{\sigma}\colon S_{n} \to S_{n}$ for a permutation pattern $\sigma$ \cite{CerbaiClaessonFerrari2020}, so that $s = s_{21}$, with a 
 related generalization of $s$ having been given by Defant and Zheng in \cite{DefantZheng2021}. Since $s$ is defined on permutations, it is natural to 
 consider generalizing this mapping using ``permutation-like'' combinatorial objects. The bases of the partition algebra $\mathscr{P}_{n}^{\xi}$ 
 generalize the bases of the order-$n$ symmetric group algebra in ways that are of interest from both a combinatorial and an algebraic perspective, and 
 this leads us to consider how partition diagrams, which index the bases of $\mathscr{P}_{n}^{\xi}$, may be used to generalize West's stack-sorting map. 
 We generalize, in this article, $s\colon S_{n} \to S_{n}$ to a mapping $\mathscr{S}\colon \mathscr{P}_{n} \to \mathscr{P}_{n}$ from the partition 
 monoid $\mathscr{P}_{n}$ to itself such that $\mathscr{S}$ restricted to permuting diagrams behaves in the same way as $s$. We then apply our lifting 
 $\mathscr{S}$ to determine an analogue of a famous result due to Knuth on $1$-stack-sortable permutations \cite[\S2.2.1]{Knuth1973}. 

 The representation theory of $\mathscr{P}_{n}^{\xi}$
 is intimately linked with that for the $n!$-dimensional symmetric group algebra. 
 Indeed, there is so much about 
 the representation theory of $\mathscr{P}_{n}^{\xi}$ and associated 
 combinatorial properties of $\mathscr{P}_{n}^{\xi}$ that are directly derived from or otherwise 
 based on the representation theory of the algebra $\text{span}(S_{n})$ and associated combinatorics 
 \cite{BenkartHalverson20171,BenkartHalverson20172,HalversonJacobson2020,HalversonRam2005,Martin1996,Martin1994,MartinWoodcock1999}. 
 The Schur--Weyl duality given by 
 how the bases of $\mathscr{P}_{n}^{\xi}$ generalize 
 the bases of $\text{span} ( S_{n})$ 
 is of importance in both combinatorial representation theory 
 and in the areas of statistical mechanics in which partition algebras had originally been defined 
 by Martin \cite{Martin2000,Martin1991,Martin1996,Martin1994} and Jones \cite{Jones1994}. 
 The foregoing considerations are representative of the extent to which partition diagrams 
 generalize permutations in a way that is of much significance in both mathematics and physics. 
 This motivates our lifting West's stack-sorting map so as to allow partition diagrams, in addition to permutations, 
 as input. 

\subsection{Preliminaries}
 To be consistent with the notation 
 in references as in 
 \cite{Defant2020Comb,Defant2021Discrete,Defant2021European,Defant2021Enumer,Defant2020Electron,Defant2022Theory,Defant2020Lothar,DefantKravitz2020,DefantZheng2021} 
 for indexing symmetric groups and maximal elements in permutations, as in \eqref{sLsRn} 
 and \eqref{smap}, 
 we let the order of a given partition algebra/monoid be denoted 
 as in the order of the symmetric group shown in \eqref{smap}. 
 In particular, following \cite{East2011}, we let the partition 
 monoid of a given order be denoted as 
 $\mathscr{P}_{n}$, and we let the corresponding partition algebra with a complex parameter $\xi$
 be denoted as $\mathscr{P}_{n}^{\xi}$. These structures are defined as follows. 

 We let $\mathscr{P}_{n}$ consist of 
 set-partitions $\mu$ of $\{ 1$, $2$, $\ldots$, $n$, $1'$, $2'$, $\ldots$, $n' \}$ that we denote with a two-line notation by analogy with 
 permutations, by aligning nodes labeled with $1$, $2$, $\ldots$, $n$ into an upper row 
 and nodes labeled with $1'$, $2'$, $\ldots$, $n'$ into a bottom row, and by forming any graph $G$ such that the 
 the set of components of $G$ equals $\mu$. 
 Graphs of this form, denoted in the manner we have specified, are referred to as \emph{partition diagrams}, 
 and two such partition diagrams are considered to be the same if the connected components are the same in both cases. 
 The components of the graph $G$ given as before are referred to as \emph{blocks}.

\begin{example}\label{illustrationblocks}
 The set-partition $\{ \{ 1, 4 \}, \{ 2, 3, 4', 5' \}, \{ 5 \}, \{ 1', 3' \}, \{ 2' \} \}$
 may be denoted as 
$$ 
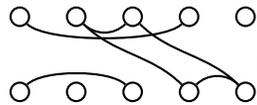
\begin{tikzpicture}[scale = 0.5,thick, baseline={(0,-1ex/2)}] 
\tikzstyle{vertex} = [shape = circle, minimum size = 7pt, inner sep = 1pt] 
\node[vertex] (G--5) at (6.0, -1) [shape = circle, draw] {}; 
\node[vertex] (G--4) at (4.5, -1) [shape = circle, draw] {}; 
\node[vertex] (G-2) at (1.5, 1) [shape = circle, draw] {}; 
\node[vertex] (G-3) at (3.0, 1) [shape = circle, draw] {}; 
\node[vertex] (G--3) at (3.0, -1) [shape = circle, draw] {}; 
\node[vertex] (G--1) at (0.0, -1) [shape = circle, draw] {}; 
\node[vertex] (G--2) at (1.5, -1) [shape = circle, draw] {}; 
\node[vertex] (G-1) at (0.0, 1) [shape = circle, draw] {}; 
\node[vertex] (G-4) at (4.5, 1) [shape = circle, draw] {}; 
\node[vertex] (G-5) at (6.0, 1) [shape = circle, draw] {}; 
\draw[] (G-2) .. controls +(0.5, -0.5) and +(-0.5, -0.5) .. (G-3); 
\draw[] (G-3) .. controls +(1, -1) and +(-1, 1) .. (G--5); 
\draw[] (G--5) .. controls +(-0.5, 0.5) and +(0.5, 0.5) .. (G--4); 
\draw[] (G--4) .. controls +(-1, 1) and +(1, -1) .. (G-2); 
\draw[] (G--3) .. controls +(-0.6, 0.6) and +(0.6, 0.6) .. (G--1); 
\draw[] (G-1) .. controls +(0.7, -0.7) and +(-0.7, -0.7) .. (G-4); 
\end{tikzpicture}, $$
 or, equivalently, as 
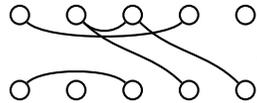
$$ \begin{tikzpicture}[scale = 0.5,thick, baseline={(0,-1ex/2)}] 
\tikzstyle{vertex} = [shape = circle, minimum size = 7pt, inner sep = 1pt] 
\node[vertex] (G--5) at (6.0, -1) [shape = circle, draw] {}; 
\node[vertex] (G--4) at (4.5, -1) [shape = circle, draw] {}; 
\node[vertex] (G-2) at (1.5, 1) [shape = circle, draw] {}; 
\node[vertex] (G-3) at (3.0, 1) [shape = circle, draw] {}; 
\node[vertex] (G--3) at (3.0, -1) [shape = circle, draw] {}; 
\node[vertex] (G--1) at (0.0, -1) [shape = circle, draw] {}; 
\node[vertex] (G--2) at (1.5, -1) [shape = circle, draw] {}; 
\node[vertex] (G-1) at (0.0, 1) [shape = circle, draw] {}; 
\node[vertex] (G-4) at (4.5, 1) [shape = circle, draw] {}; 
\node[vertex] (G-5) at (6.0, 1) [shape = circle, draw] {}; 
\draw[] (G-2) .. controls +(0.5, -0.5) and +(-0.5, -0.5) .. (G-3); 
\draw[] (G-3) .. controls +(1, -1) and +(-1, 1) .. (G--5); 
\draw[] (G--4) .. controls +(-1, 1) and +(1, -1) .. (G-2); 
\draw[] (G--3) .. controls +(-0.6, 0.6) and +(0.6, 0.6) .. (G--1); 
\draw[] (G-1) .. controls +(0.7, -0.7) and +(-0.7, -0.7) .. (G-4); 
\end{tikzpicture}. $$
\end{example}

 Following \cite{HalversonRam2005}, we let the underlying field for partition algebras and symmetric group algebras be $\mathbb{C}$, as $\mathbb{C}$ 
 being algebraically closed is of direct relevance in terms of the study of the semisimple structure for partition algebras \cite{HalversonRam2005}. We 
 may let the symmetric group algebra of order $n$ be denoted by taking the linear span $\text{span}(S_{n})$, or, more explicitly, 
 $\text{span}_{\mathbb{C}} \{ \sigma : \sigma \in S_{n} \}$, of the symmetric group $S_{n}$. 
 
 For two partition diagrams $d_{1}$ and $d_{2}$, we place $d_{1}$ on top of $d_{2}$
 so that the bottom nodes of $d_{1}$ overlap with the top nodes of $d_{2}$, 
 then we remove the central row in this concatenation $d_{1} \ast d_{2}$ in such a way so as to preserve the relation
 given by topmost nodes being in the same component as bottommost nodes in $d_{1} \ast d_{2}$. 
 We then let $d_{1} \circ d_{2}$ denote the graph thus obtained from this concatenation. 
 This is the underlying multiplicative operation for partition monoids. 

\begin{example}
 Borrowing an example from \cite{East2011}, we let $d_{1}$ be as in Example \ref{illustrationblocks}, and we let $d_{2}$ be as below: 
 $$ 
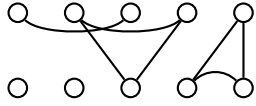
\begin{tikzpicture}[scale = 0.5,thick, baseline={(0,-1ex/2)}] 
\tikzstyle{vertex} = [shape = circle, minimum size = 7pt, inner sep = 1pt] 
\node[vertex] (G--5) at (6.0, -1) [shape = circle, draw] {}; 
\node[vertex] (G--4) at (4.5, -1) [shape = circle, draw] {}; 
\node[vertex] (G-5) at (6.0, 1) [shape = circle, draw] {}; 
\node[vertex] (G--3) at (3.0, -1) [shape = circle, draw] {}; 
\node[vertex] (G-2) at (1.5, 1) [shape = circle, draw] {}; 
\node[vertex] (G-4) at (4.5, 1) [shape = circle, draw] {}; 
\node[vertex] (G--2) at (1.5, -1) [shape = circle, draw] {}; 
\node[vertex] (G--1) at (0.0, -1) [shape = circle, draw] {}; 
\node[vertex] (G-1) at (0.0, 1) [shape = circle, draw] {}; 
\node[vertex] (G-3) at (3.0, 1) [shape = circle, draw] {}; 
\draw[] (G-5) .. controls +(0, -1) and +(0, 1) .. (G--5); 
\draw[] (G--5) .. controls +(-0.5, 0.5) and +(0.5, 0.5) .. (G--4); 
\draw[] (G--4) .. controls +(0.75, 1) and +(-0.75, -1) .. (G-5); 
\draw[] (G-2) .. controls +(0.6, -0.6) and +(-0.6, -0.6) .. (G-4); 
\draw[] (G-4) .. controls +(-0.75, -1) and +(0.75, 1) .. (G--3); 
\draw[] (G--3) .. controls +(-0.75, 1) and +(0.75, -1) .. (G-2); 
\draw[] (G-1) .. controls +(0.6, -0.6) and +(-0.6, -0.6) .. (G-3); 
\end{tikzpicture}. $$
 We may verify that the monoid product $d_{1} \circ d_{2}$, in this case, 
 is as below \cite{East2011}: 
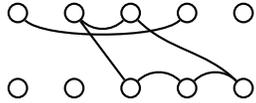
 $$ \begin{tikzpicture}[scale = 0.5,thick, baseline={(0,-1ex/2)}] 
\tikzstyle{vertex} = [shape = circle, minimum size = 7pt, inner sep = 1pt] 
\node[vertex] (G--5) at (6.0, -1) [shape = circle, draw] {}; 
\node[vertex] (G--4) at (4.5, -1) [shape = circle, draw] {}; 
\node[vertex] (G--3) at (3.0, -1) [shape = circle, draw] {}; 
\node[vertex] (G-2) at (1.5, 1) [shape = circle, draw] {}; 
\node[vertex] (G-3) at (3.0, 1) [shape = circle, draw] {}; 
\node[vertex] (G--2) at (1.5, -1) [shape = circle, draw] {}; 
\node[vertex] (G--1) at (0.0, -1) [shape = circle, draw] {}; 
\node[vertex] (G-1) at (0.0, 1) [shape = circle, draw] {}; 
\node[vertex] (G-4) at (4.5, 1) [shape = circle, draw] {}; 
\node[vertex] (G-5) at (6.0, 1) [shape = circle, draw] {}; 
\draw[] (G-2) .. controls +(0.5, -0.5) and +(-0.5, -0.5) .. (G-3); 
\draw[] (G-3) .. controls +(1, -1) and +(-1, 1) .. (G--5); 
\draw[] (G--5) .. controls +(-0.5, 0.5) and +(0.5, 0.5) .. (G--4); 
\draw[] (G--4) .. controls +(-0.5, 0.5) and +(0.5, 0.5) .. (G--3); 
\draw[] (G--3) .. controls +(-0.75, 1) and +(0.75, -1) .. (G-2); 
\draw[] (G-1) .. controls +(0.7, -0.7) and +(-0.7, -0.7) .. (G-4); 
\end{tikzpicture}. $$
\end{example}

 For a complex parameter $\xi$, we endow the $\mathbb{C}$-span of $\mathscr{P}_{n}$ with a multiplicative binary operation as 
 follows: $ d_1 d_2 = \xi^{\ell} d_{1} \circ d_{2}$, 
 where $\ell$ denotes the number of components contained entirely in the middle row of $d_{1} \ast d_{2}$. 
 By extending this binary operation linearly, this gives us the underlying multiplicative operation 
 for the structure known as the partition algebra, which is denoted as $\mathscr{P}_{n}^{\xi}$. 
 The set of all elements in $\mathscr{P}_{n}$, 
 as elements in $\mathscr{P}_{n}^{\xi}$, is referred to as the \emph{diagram basis} of $\mathscr{P}_{n}^{\xi}$. 

 Although the algebraic structure of $\mathscr{P}_{n}^{\xi}$ will not be used directly in this article, 
 it is useful to define $\mathscr{P}_{n}^{\xi}$ explicitly, since we are to heavily make use of an algebra homomorphism 
 defined on partition algebras and introduced in \cite{Campbell2022}, 
 and it is often convenient to use notation and terminology associated with $\mathscr{P}_{n}^{\xi}$ or its algebraic structure, 
 e.g., by referring to the diagram basis of $\mathscr{P}_{n}^{\xi}$. 

 For a partition diagram $\pi$, the \emph{propagation number} of $\pi$ refers to the number of blocks 
 in $\pi$ containing at least one vertex in the top row and at least on 
 vertex in the bottom row. 
 So, there is a clear bijection between the set of all permutations in $S_{n}$ 
 and the set of all diagrams in $\mathscr{P}_{n}$ 
 that are of propagation number $n$. 
 In our lifting West's stack-sorting map to the partition monoid, 
 to be consistent with two-line notation for permutations and with \eqref{pLnR}, 
 we let a permutation 
\begin{equation}\label{pcolon}
 p\colon \{ 1, 2, \ldots, n \} \to \{ 1, 2, \ldots, n \}
\end{equation}
 be in correspondence with the partition diagram $\pi$ given by the set-partition 
\begin{equation}\label{embedding}
 \{ \{ 1, p(1)' \}, \{ 2, p(2)' \}, \ldots, \{ n, p(n)' \} \}, 
\end{equation}
 although it is common to instead let a permutation $p$ be mapped to the partition diagram obtained by reflecting $\pi$ vertically. 

\section{A lift of West's stack-sorting map}\label{sectionlift}
 As Defant and Kravtiz explain in \cite{DefantKravitz2020}, 
 there is a matter of ambiguity in
 the problem of lifting the mapping $s$ so as to allow words allowing \emph{repeated} characters, as opposed to permutations
 written as a word in the manner indicated in \eqref{pLnR}, as the argument, for an extension of $s$. 
 We encounter similar kinds of problems in terms of the problem of lifting $s$ so as to allow partition diagrams, 
 as opposed to permutations written as permutation diagrams, as the argument, again for a lifting of $s$. 
 This is illustrated as below. 

\begin{example}\label{exampleambiguity}
 According to the embedding indicated in \eqref{embedding}, 
 we let the permutation $\left(\begin{smallmatrix} 
1 & 2 & 3 \\ 
 2 & 3 & 1 
 \end{smallmatrix}\right)$ be written as 
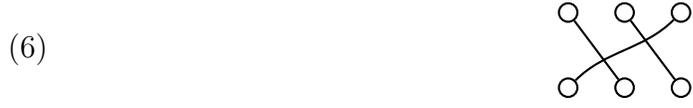
\begin{equation}\label{diagram231}
 \begin{tikzpicture}[scale = 0.5,thick, baseline={(0,-1ex/2)}] 
\tikzstyle{vertex} = [shape = circle, minimum size = 7pt, inner sep = 1pt] 
\node[vertex] (G--3) at (3.0, -1) [shape = circle, draw] {}; 
\node[vertex] (G-2) at (1.5, 1) [shape = circle, draw] {}; 
\node[vertex] (G--2) at (1.5, -1) [shape = circle, draw] {}; 
\node[vertex] (G-1) at (0.0, 1) [shape = circle, draw] {}; 
\node[vertex] (G--1) at (0.0, -1) [shape = circle, draw] {}; 
\node[vertex] (G-3) at (3.0, 1) [shape = circle, draw] {}; 
\draw[] (G-2) .. controls +(0.75, -1) and +(-0.75, 1) .. (G--3); 
\draw[] (G-1) .. controls +(0.75, -1) and +(-0.75, 1) .. (G--2); 
\draw[] (G-3) .. controls +(-1, -1) and +(1, 1) .. (G--1); 
\end{tikzpicture}. 
\end{equation}
 If we visualize the block $\{ 2, 3' \}$ being removed and placed on the right of a resultant configuration, by analogy with how the permutation \eqref{pLnR} 
 gets mapped to the right-hand side of \eqref{sLsRn} under the application of West's stack-sorting map, 
 there is a clear way how this removal ``splits'' the permutation diagram in \eqref{diagram231}
 into a left and a right configuration by direct analogy with 
 \eqref{pLnR}, and part of the purpose of our procedure in Section \ref{subsectionprocedure}
 is to formalize this idea in a way that may be extended to all partition diagrams. 
 In contrast, if we remove, for example, the rightmost block of 
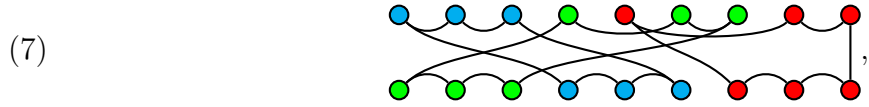
\begin{equation}\label{clearillustrationmain}
\begin{tikzpicture}[scale = 0.5,thick, baseline={(0,-1ex/2)}] 
\tikzstyle{vertex} = [shape = circle, minimum size = 7pt, inner sep = 1pt] 
\node[vertex] (G--9) at (12.0, -1) [shape = circle, draw,fill=red] {}; 
\node[vertex] (G--8) at (10.5, -1) [shape = circle, draw,fill=red] {}; 
\node[vertex] (G--7) at (9.0, -1) [shape = circle, draw,fill=red] {}; 
\node[vertex] (G-5) at (6.0, 1) [shape = circle, draw,fill=red] {}; 
\node[vertex] (G-8) at (10.5, 1) [shape = circle, draw,fill=red] {}; 
\node[vertex] (G-9) at (12.0, 1) [shape = circle, draw,fill=red] {}; 
\node[vertex] (G--6) at (7.5, -1) [shape = circle, draw,fill=cyan] {}; 
\node[vertex] (G--5) at (6.0, -1) [shape = circle, draw,fill=cyan] {}; 
\node[vertex] (G--4) at (4.5, -1) [shape = circle, draw,fill=cyan] {}; 
\node[vertex] (G-1) at (0.0, 1) [shape = circle, draw,fill=cyan] {}; 
\node[vertex] (G-2) at (1.5, 1) [shape = circle, draw,fill=cyan] {}; 
\node[vertex] (G-3) at (3.0, 1) [shape = circle, draw,fill=cyan] {}; 
\node[vertex] (G--3) at (3.0, -1) [shape = circle, draw,fill=green] {}; 
\node[vertex] (G--2) at (1.5, -1) [shape = circle, draw,fill=green] {}; 
\node[vertex] (G--1) at (0.0, -1) [shape = circle, draw,fill=green] {}; 
\node[vertex] (G-4) at (4.5, 1) [shape = circle, draw,fill=green] {}; 
\node[vertex] (G-6) at (7.5, 1) [shape = circle, draw,fill=green] {}; 
\node[vertex] (G-7) at (9.0, 1) [shape = circle, draw,fill=green] {}; 
\draw[] (G-5) .. controls +(0.7, -0.7) and +(-0.7, -0.7) .. (G-8); 
\draw[] (G-8) .. controls +(0.5, -0.5) and +(-0.5, -0.5) .. (G-9); 
\draw[] (G-9) .. controls +(0, -1) and +(0, 1) .. (G--9); 
\draw[] (G--9) .. controls +(-0.5, 0.5) and +(0.5, 0.5) .. (G--8); 
\draw[] (G--8) .. controls +(-0.5, 0.5) and +(0.5, 0.5) .. (G--7); 
\draw[] (G--7) .. controls +(-1, 1) and +(1, -1) .. (G-5); 
\draw[] (G-1) .. controls +(0.5, -0.5) and +(-0.5, -0.5) .. (G-2); 
\draw[] (G-2) .. controls +(0.5, -0.5) and +(-0.5, -0.5) .. (G-3); 
\draw[] (G-3) .. controls +(1, -1) and +(-1, 1) .. (G--6); 
\draw[] (G--6) .. controls +(-0.5, 0.5) and +(0.5, 0.5) .. (G--5); 
\draw[] (G--5) .. controls +(-0.5, 0.5) and +(0.5, 0.5) .. (G--4); 
\draw[] (G--4) .. controls +(-1, 1) and +(1, -1) .. (G-1); 
\draw[] (G-4) .. controls +(0.6, -0.6) and +(-0.6, -0.6) .. (G-6); 
\draw[] (G-6) .. controls +(0.5, -0.5) and +(-0.5, -0.5) .. (G-7); 
\draw[] (G-7) .. controls +(-1, -1) and +(1, 1) .. (G--3); 
\draw[] (G--3) .. controls +(-0.5, 0.5) and +(0.5, 0.5) .. (G--2); 
\draw[] (G--2) .. controls +(-0.5, 0.5) and +(0.5, 0.5) .. (G--1); 
\draw[] (G--1) .. controls +(1, 1) and +(-1, -1) .. (G-4); 
\end{tikzpicture}, 
\end{equation}
 where the above colouring is to distinguish the blocks in the partition diagram  depicted, 
 then the block highlighted in green is not strictly to the left of the rightmost block, 
 since there are green nodes to the left and right of a red node in the upper row. 
\end{example}

 To lift the recursive definition for $s$ indicated in \eqref{sLsRn} so as to allow members of 
 $\mathscr{P}_{n}$ as the input for an extension or lifting of $s$, 
 and to deal with the matter of ambiguity explained in Example \ref{exampleambiguity}, 
 we would want to mimic \eqref{sLsRn} by allowing the possibility of ``middle'' 
 configurations, as opposed to the left-right dichotomy indicated in \eqref{pLnR}
 and \eqref{sLsRn}. We formalize this idea in Section \ref{subsectionprocedure}. 

\subsection{A procedure for stack-sorting partition diagrams}\label{subsectionprocedure}
 Define $\mathscr{S}_{n}\colon \mathscr{P}_{n} \to \mathscr{P}_{n}$ according to the following procedure, for an arbitrary partition diagram $\pi$ 
 in $\mathscr{P}_{n}$. 
 We order the bottom nodes of a given partition diagram in $\mathscr{P}_{n}$
 in the natural way, with $1' < 2' < \cdots < n'$. 
 For the sake of convenience, we may write $\mathscr{S} = \mathscr{S}_{n}$. 

\begin{enumerate}

\item If there are no propagating blocks in $\pi$, skip the below steps involving propagating blocks. 

\item Take the largest bottom node of $\pi$ that is part of a propagating block $B$. 
 This propagating block separates $\pi$ into three (possibly empty) classes of configurations according to how 
 the top nodes of $\pi$ are separated by removing $B$. 
 Explicitly, we define $\mathscr{L}$, $\mathscr{M}_{1}$, $\mathscr{M}_{2}$, $\ldots$, $\mathscr{M}_{\mu}$, 
 and $\mathscr{R}$ as follows, by direct analogy with \eqref{pLnR}. 
 We may denote $\mathscr{L}$, $\mathscr{M}_{i}$, and $\mathscr{R}$ as diagrams in $\mathscr{P}_{n}$. 
 The blocks of $\mathscr{L}$ (resp.\ $\mathscr{R}$) consist of either: 

$\bullet$ Any blocks of $\pi$ with upper nodes 
 that are all strictly to the left (resp.\ right) 
 of all of the upper nodes of $B$; 

$\bullet$ Any non-propagating blocks of $\pi$ on the bottom row of $\pi$ 
 with nodes 
 that are all strictly to the left (resp.\ right) 
 of all of the lower nodes of $B$; or 

$\bullet$ Singleton blocks. 

The blocks of expressions of the form $\mathscr{M}_{i}$ consist of either: 

$\bullet$ Any blocks of $\pi$ that do not satisfy either of the first two bullet points listed above; or 

$\bullet$ Singleton blocks. 

 We order $\mathscr{M}_{1} < \mathscr{M}_{2} < \cdots < \mathscr{M}_{\mu}$ according to the 
 ordering of the minimal elements, subject the ordering whereby $1' < 2' < \cdots < n' < 1 < 2 < \cdots < n$. 
 We then let $B'$ denote the partition diagram obtained from $B$ by adding any singleton nodes 
 so as to form a partition diagram in $\mathscr{P}_{n}$. 
 With this setup, we write 
\begin{equation}\label{maindecomposition}
 \mathscr{S}(\pi) = \mathscr{S}(\mathscr{L}) \odot \mathscr{S}(\mathscr{M}_{1}) 
 \odot \cdots \odot \mathscr{S}(\mathscr{M}_{\mu}) \odot \mathscr{S}(\mathscr{R}) \odot B', 
\end{equation}
 where the associative binary relation $ \odot $ is to later be defined. 

\item Repeatedly apply the above step wherever possible, i.e., to the expressions 
 in or derived from \eqref{maindecomposition}
 given by $\mathscr{S}$ evaluated at a partition diagram. 

\item The above steps yield a $\odot$-product of expressions of the forms 
\begin{equation}\label{unorderedodot}
 \mathscr{S}(N_{1}), \cdots, \mathscr{S}(N_{m_{1}}) \ \ \ \text{and} \ \ \ 
 B_{1}', \cdots, B_{m_{2}}', 
\end{equation}
 not necessarily in this order, for non-propagating diagrams $N_{i}$, 
 and where each expression of the form $B_{j}'$ is a partition diagram with exactly one propagating block
 and with singleton blocks anywhere else. 
 The factors of the aforementioned $\odot$-product
 indicated in \eqref{unorderedodot} are ordered in the following way: 
 $B_{j}'$ is the $j^{\text{th}}$ 
 factor of the form $B_{\kappa}'$ appearing in this $\odot$-product, 
 and $\mathscr{S}(N_{i})$ is the $i^{\text{th}}$ factor of the form 
 $\mathscr{S}(N_{\kappa})$ appearing in this $\odot$-product. 
 The operation $\odot$ indicates that the following is to be applied. 
 We label the top nodes in the propagating block in $B_{1}'$
 from left to right with consecutive integers starting with $1$, 
 and we then label the top nodes in the propagating block in $B_{2}'$ 
 with consecutive integers (starting with $1$ plus the number of top nodes in the propagating block in $B_{1}'$), 
 and we continue in this manner. We then continue with this labeling, by 
 labeling any non-propagating blocks of size greater than $1$ 
 in the top row of the $N_{i}$-expressions, in order of the nodes as they appear 
 among consecutive $N_{i}$-diagrams. If there are any unused labels for the top row, label 
 singleton blocks in the upper row with these leftover labels. 
 However, for the bottom nodes of any non-singleton blocks from $\pi$, we let these bottom nodes keep their original labelings 
 (and if necessary we add in singleton blocks in a bottom row to form a partition diagram based on the preceding steps). 

\item As indicated above, the above steps produce an element in $\mathscr{P}_{n}$. We set $\mathscr{S}(\pi)$ 
 to be this element. 

\end{enumerate}

\begin{example}\label{mainillustration}
 Let $\pi$ denote the partition diagram 
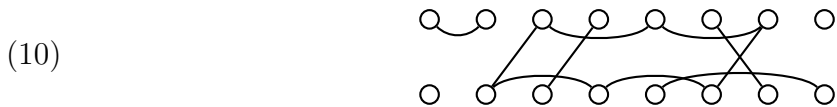
\begin{equation}\label{20221205617PM2A}
 
\end{equation}
 in $\mathscr{P}_{8}$ corresponding to the following set-partition: 
 $$ \{ \{1,2\},\{3,5,7,-2,-4,-6\},\{4,-3\},\{6,-7\},\{8\},\{-1\},\{-5,-8\}\}.$$ 
 The largest bottom node of $\pi$ that is part of a propagating block $B$, in this case, is $7'$, 
 and $B$ is $\{ 6, 7' \}$. For the sake of clarity, we highlight this block in
 the following manner. 
 $$ %
 $$
 The blocks of $\mathscr{L}$ are highlighted in cyan as below (with the propagating block $B$ again highlighted
 in a different colour for the sake of clarity). 
 $$ %
 $$
 The blocks of $\mathscr{M}$ are highlighted in green as below. 
 $$ %
 $$
 We see that $\mathscr{R}$ consists of only one block, which is highlighted in orange, as below. 
 $$ %
 $$
 So, according to our lifting of West's stack-sorting map, we obtain 
 the following, where singleton nodes coloured black indicate that these singleton nodes have been ``added in'' 
 according to the above given procedure. 
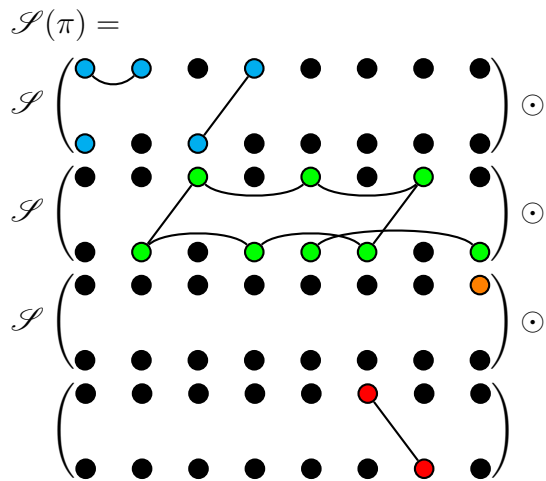
\begin{align*}
 & \mathscr{S}(\pi) = \\
 & \mathscr{S}\left( %
 
 \right)
\end{align*}
 Now, let us repeat the given procedure to the first $\mathscr{S}$-factor on the right-hand side of the above equality, 
 and then to the resultant $\mathscr{S}$-factor involving an argument with a block of size $6$.
 This results in the following equality. 
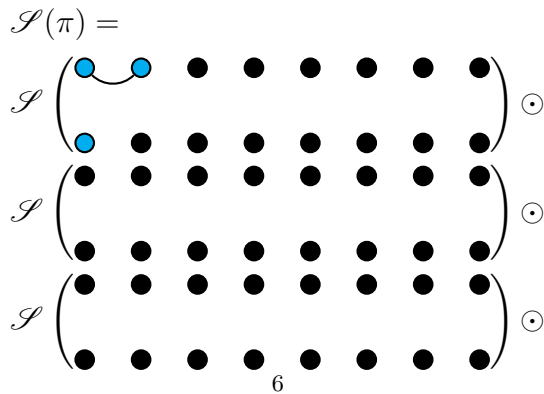
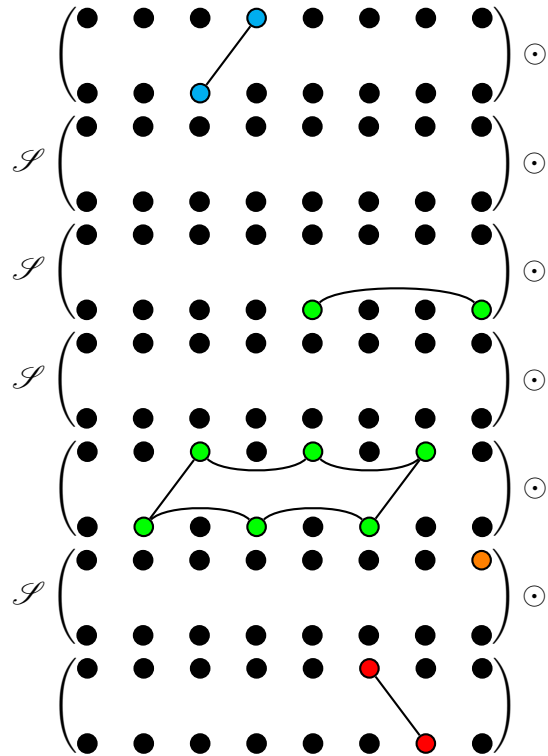
\begin{align*}
 & \mathscr{S}(\pi) = \\
 & \mathscr{S}\left( %
 
 \right)
\end{align*}
 So, we have determined a $\odot$-product of expressions of the forms 
 indicated in \eqref{unorderedodot}. Now, apply the labeling indicated as follows, according to our procedure. 
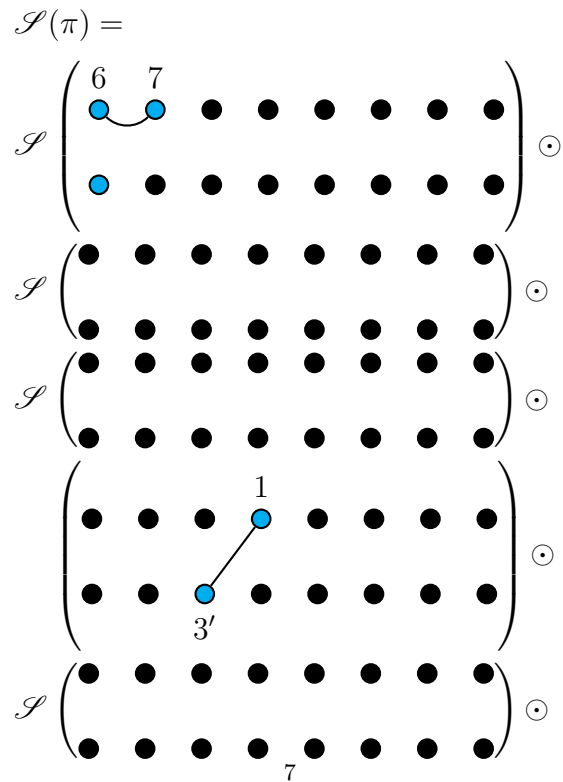
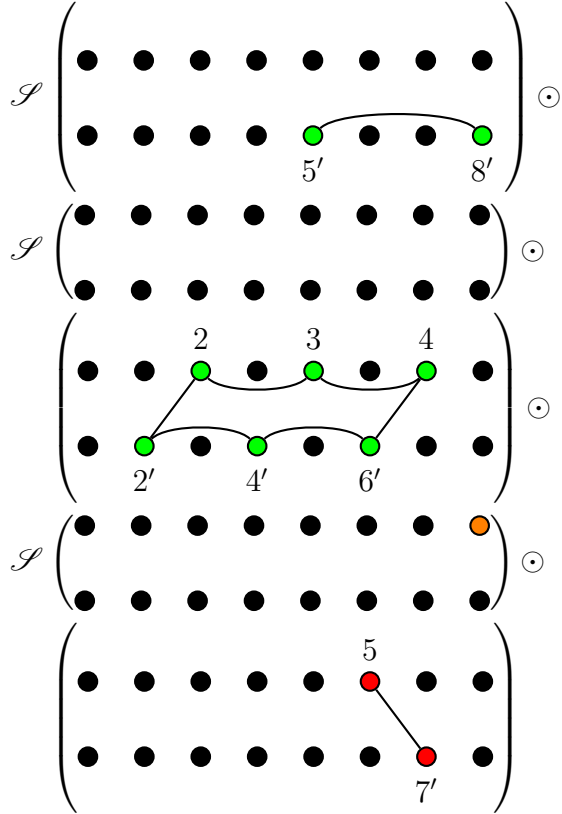
\begin{align*}
 & \mathscr{S}(\pi) = \\
 & \mathscr{S}\left( %
 \right)
\end{align*}
 So, this shows us how the partition diagram in \eqref{20221205617PM2A} gets mapped to 
 $$ %
 $$ according to our sorting map. 
\end{example}

\subsection{Sorting permuting diagrams} 
 Since we claim that our sorting map 
 $\mathscr{S}\colon \mathscr{P}_{n} \to \mathscr{P}_{n}$ lifts \eqref{smap}, 
 it would be appropriate to formalize and prove this claim, as in Theorem \ref{theoremlift}
 below and our proof of this Theorem. We proceed to briefly review some preliminaries 
 concerning Theorem \ref{theoremlift}. 

 Our convention for mapping permutations to permutation diagrams, as indicated in 
 \eqref{embedding}, is to be used consistently throughout our article. 
 Again, this convention is consistent with the notation for West's stack-sorting map indicated in 
 \eqref{pLnR}, since it mimics two-line notation for permutations. 
 We are to extend, as below, this embedding so as to be applicable to expressions as in $L$ and $R$ in 
 the permutation decomposition shown in \eqref{pLnR}. 

 The \emph{rook algebra} is a subalgebra of $\mathscr{P}_{n}^{\xi}$ that is spanned by partial permutations. Partial permutations are diagrams that 
 consist of blocks of size $1$ and blocks of size $2$ that consist of a vertex in the top row and a vertex in the bottom row \cite{HalversonJacobson2020}. 
 Following \cite{Xiao2016}, the expression $\text{Rd}_{k}$ denotes the set of all rook $k$-diagrams.
 Similarly, for each $r \in \mathbb{Z}$ satisfying 
 $0 \leq r \leq k$, the expression $\text{Rd}_{k}[r]$ 
 denotes the set of rook $k$-diagrams with precisely $r$ singleton vertices in each row. 
 Given a permutation decomposition of the form indicated in \eqref{pLnR}, 
 we can identify $L$ (resp.\ $R$) with a partial permutation 
 such that the primed elements in any $2$-blocks in this partial permutation 
 are the primed versions of any numbers in $L$ (resp.\ $R$) 
 and in such a way so as to agree with the two-line notation indicated in \eqref{pLnR}. 
 Explicitly, if $L$ is empty, we let it be mapped to the partition diagram in $\mathscr{P}_{n}$
 consisting of singleton blocks, and if we write $L = l_{1} l_{2} \cdots l_{\ell(L)}$, 
 we may identify $L$ with the partition diagram with $2$-blocks of the forms 
 $$ \{ p^{-1}(l_{1}), l_{1}' \}, \{ p^{-1}(l_{2}), l_{2}' \}, 
 \ldots, \{ p^{-1}(l_{\ell(L)}), l_{\ell(L)}' \} $$
 and with singleton blocks everywhere else, and similarly for $R$. 
 This agrees with our convention indicated in \eqref{embedding} 
 for embedding $S_{n}$ into $\mathscr{P}_{n}$. 

 Again with reference to the permutation decomposition in \eqref{pLnR}
 we may identify the expression $n$ with a partial permutation and 
 in a similar fashion as above, with the understanding that this expression 
 is part of the concatenation in \eqref{pLnR}. 
 Explicitly, we would identify $n$ with the partition diagram given by the set-partition
\begin{align}
 & \{ \{ p^{-1}(n), n' \} \} \cup \label{202212051128100perPM1A} \\
 & \{ \{ x \} : x \in \mathbb{N}, 1 \leq x \leq n, x \neq p^{-1}(n) \} \cup \label{202212051128100perPM2A} \\
 & \{ \{ x \} : x \in \mathbb{N}, 1 \leq x \leq n, x \neq n \}. \label{202212051128100perPM3A} 
\end{align}

 As illustrated in Example \ref{mainillustration}, for permutation diagrams 
 $\pi_{1}$ and $\pi_{2}$, if $\pi_{2}$ consists entirely of singleton nodes then
 $\mathscr{S}(\pi_{1}) \odot \mathscr{S}(\pi_{2}) = \mathscr{S}(\pi_{2}) \odot \mathscr{S}(\pi_{1}) 
 = \mathscr{S}(\pi_{1})$. This algebraic property is to be used in our proof of Theorem \ref{theoremlift}. 

\begin{theorem}\label{theoremlift}
 Let $p$ be a permutation in $S_{n}$, and let $\pi$ denote the corresponding partition diagram 
 according to \eqref{embedding}. Then $\mathscr{S}(\pi) $ equals the partition diagram 
 corresponding to $s(p)$. 
\end{theorem}

\begin{proof}
 As above, we let $p \in S_{n}$, writing $p\colon \{ 1, 2, \ldots, n \} \to \{ 1, 2, \ldots, n \}$. 
 As above, we let $\pi$ denote the partition diagram corresponding to 
 \eqref{embedding}, i.e., so that the set of blocks of $\pi$ is \eqref{embedding}. 
 With respect to the notation in \eqref{maindecomposition}, 
 the expression $ \mathscr{S}(\mathscr{L})$ (resp.\ $\mathscr{S}(\mathscr{R})$) is equal to $\mathscr{S}$ evaluated 
 at a diagram obtained from $\pi$ by taking any blocks 
 consisting of a bottom node that is labeled the primed version 
 of a number to the left (resp.\ right) of $n$ in the sense indicated in \eqref{pLnR}, 
 i.e., a number in $L$ (resp.\ $R$) according to the notation in \eqref{pLnR} for a permutation $p \in S_{n}$. 
 Since $\pi$ is a permutation diagram,   any   $\mathscr{M}$-expression    consists    entirely of singleton nodes.   
   So, the product in \eqref{maindecomposition}    reduces to   
\begin{equation}\label{202212100100051110PM1A}
 \mathscr{S}(\pi) = \mathscr{S}(\mathscr{L}) \odot \mathscr{S}(\mathscr{R}) \odot B'. 
\end{equation}
 As indicated above, $\mathscr{L}$ (resp.\ $\mathscr{R}$) is precisely the partition diagram given by 
 embedding $L$ (resp.\ $R$) into $\mathscr{P}_{n}$. Similarly, $B'$
 is precisely the partition diagram given by embedding the factor $n$ in \eqref{pLnR}, 
 in the manner indicated in \eqref{202212051128100perPM1A}--\eqref{202212051128100perPM3A}. 
 So, letting it be understood that the factors in 
 \eqref{pLnR} and the left-hand side of \eqref{pLnR} 
 may be identified with their respective embeddings, 
 we find that the $\odot$-product in \eqref{202212100100051110PM1A} may be written as 
\begin{equation*}
 \mathscr{S}(p) = \mathscr{S}(L) \odot \mathscr{S}(R) \odot n. 
\end{equation*}
 By comparing both sides of the above decomposition with both sides of 
 the decomposition in \eqref{sLsRn}, 
 an inductive argument provides us with the desired result. 
\end{proof}

 Since we have lifted West's stack-sorting map so as to allow partition diagrams
 in addition to permuting diagrams as input, 
 this leads us to consider how fundamental properties concerning the $s$-map 
 \eqref{smap} may be ``translated'' according to our lifting. 
 In particular, we are led to desire to generalize Knuth's classic result 
 that a permutation is $1$-stack-sortable iff it avoids the pattern 231. 

The intricate combinatorial structures and behaviours associated with our    lifting $\mathscr{S}$   
    of West's stack-sorting map are investigated in Section \ref{sectionPattern} below. 
 Such investigations are inspired by the extent of mathematical interest 
 concerning the beautiful combinatorics, at both a structural and enumerative level, associated with an important variant of the $s$-map 
 referred to as \emph{pop-stack sorting} for permutations, 
 with reference to the work of Asinowski et al.\ \cite{AsinowskiBanderierBilleyHacklLinusson2019}. 

\section{Pattern avoidance}\label{sectionPattern}
 We again let $p \in S_{n}$ be written as a mapping in the manner indicated in \eqref{pcolon}. In the context of the study of pattern avoidance in 
 permutations, it is often desirable to denote $p$ using an $n \times n$ grid in the Cartesian plane. 
 For example, if we begin by letting a permutation $p \in S_{3}$
 be denoted in the manner indicated in \eqref{pLnR}, 
 let us write, for example $p = 312$, which may be denoted using the $n \times n$ grid below. 
$$\begin{array}{|ccc|}
 \hline
 \bullet & \null & \null \\ 
 \null & \null & \bullet \\ 
 \null & \bullet & \null \\ 
 \hline
 \end{array}
$$
 More generally, we may denote a permutation $p = p(1) p(2) \cdots p(n)$
 with an $n \times n$ array such that an $(i, j)$-entry is nonempty 
 if and only if $(i, j)$ is of the form $(k, n - p^{-1}(k) +1 )$ for some index $k$, 
 and where an $(k, n - p^{-1}(k) +1 )$-entry is highlighted in some way, as above. 
 We say that a permutation $p = p(1) p(2) \cdots p(n)$ contains the pattern $231$
 if there exist indices $k_{1}$, $k_{2}$, and $k_{3}$ such that $k_{1} < k_{2} < k_{3}$
 and such that the entries 
\begin{align}
 (k_{1}, n - p^{-1}(k_{1}) +1 ), \label{kayy1} \\
 (k_{2}, n - p^{-1}(k_{2}) +1 ), \label{kayy2} \\
 (k_{3}, n - p^{-1}(k_{3}) +1 ), \label{kayy3} 
\end{align}
 satisfy the following in the $n \times n$ array we use to denote $n$: 
 Point \eqref{kayy1} is strictly below \eqref{kayy2} and is strictly above point \eqref{kayy3}. 
 In other words, the $n \times n$ array corresponding to $p$ contains a pattern 
 that is equivalent, in the sense that we have specified, to the following configuration. 
$$\begin{array}{|ccc|}
 \hline
 \null & \bullet & \null \\ 
 \bullet & \null & \null \\ 
 \null & \null & \bullet \\ 
 \hline
 \end{array}
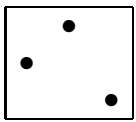
$$
 A permutation is said to \emph{avoid} the pattern $231$ if it does not contain this pattern. 
 In a similar fashion, we may characterize or define a permutation that avoids the pattern $12$
 as a \emph{decreasing} permutation, i.e., a permutation of the following form. 
$$\begin{array}{|cccc|}
 \hline
 \bullet & \null & \null & \null \\ 
 \null & \bullet & \null & \null \\ 
 \null & \null & \ddots & \null \\ 
 \null & \null & \null & \bullet \\ 
 \hline
 \end{array}
$$
 Correspondingly, an \emph{increasing} permutation is of the form shown below, i.e., it is an identity permutation. 
$$\begin{array}{|cccc|}
 \hline
 \null & \null & \null & \bullet \\ 
 \null & \null & \cdot^{\displaystyle\cdot^{\displaystyle\cdot}} & \null \\ 
 \null & \bullet & \null & \null \\ 
 \bullet & \null & \null & \null \\ 
 \hline
 \end{array}
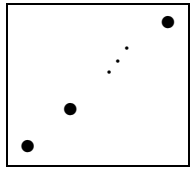
$$
 A permutation $p$ is said to be \emph{$t$-stack-sortable} if 
\[   \underbrace{s \circ s \circ \cdots \circ s}_{t}(p)    \]  
    is increasing; see \cite{Defant2021Discrete,DefantKravitz2020} for related research that has inspired much about this article.  

    The concept of a $231$-avoiding  permutation is of considerable interest for the purposes of this article,     so it is worthwhile to illustrate a permutation of this   
  form.    In this regard, we see that the permutation corresponding to   
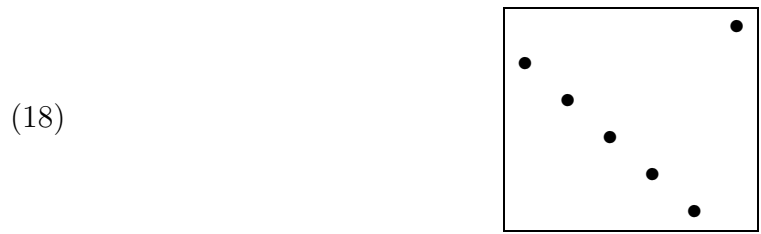
\begin{equation}\label{bulletsavoid231}
\begin{array}{|cccccc|}
 \hline
 \null & \null & \null & \null & \null & \bullet \\ 
 \bullet & \null & \null & \null & \null & \null \\ 
 \null & \bullet & \null & \null & \null & \null \\ 
 \null & \null & \bullet & \null & \null & \null \\ 
 \null & \null & \null & \bullet & \null & \null \\ 
 \null & \null & \null & \null & \bullet & \null \\ 
 \hline
 \end{array}
\end{equation}
 avoids the pattern $231$. So, according to Knuth's characterization of $1$-stack-sortable permutations, we would expect  
 the permutation illustrated in \eqref{bulletsavoid231}  to be 
   $1$-stack-sortable. It is worthwhile for our purposes to illustrate this. 

\begin{example}
 Being consistent with the notation in \eqref{pLnR}, the permutation $p \in S_{6}$
 depicted in \eqref{bulletsavoid231} may be written as $543216$. 
 According to the recursion shown in \eqref{sLsRn} for West's stack-sorting map, 
 we obtain the following: 
\begin{align*}
 s(p) & = s(543216) \\
 & = s(54321) 6 \\ 
 & = s(4321) 5 6 \\ 
 & = s(321) 4 5 6 \\ 
 & = s(21) 3 4 5 6 \\ 
 & = s(1) 2 3 4 5 6 \\ 
 & = 1 2 3 4 5 6. 
\end{align*}
 So, since $s(p)$ is the identity permutation in $S_{6}$, we have that $p$ is $1$-stack-sortable. 
\end{example}

 The main goal of our current section is to generalize the following groundbreaking result due to Knuth, 
 which, as indicated in \cite{Defant2021Discrete}, was the starting point 
 for the study of both stack-sorting and pattern avoidance in permutations. 

\begin{theorem}\label{theoremKnuth1973}
 (Knuth, 1973) A permutation is $1$-stack sortable iff it avoids $231$ (cf.\ \cite{Defant2021Discrete}). 
\end{theorem}

 So, in view of our lifting $\mathscr{S}$ of the mapping $s$, 
 what would be an appropriate way of generalizing Theorem \ref{theoremKnuth1973} 
 according to the mapping $\mathscr{S}$? In particular, how can the concept of 
 ``$1$-stack-sortability'' be translated in a meaningful way so as to be applicable with respect to $\mathscr{S}$? 
 To answer these questions, we are to make use of a morphism on partition algebras that we had previously
 applied in a representation-theoretic context \cite{Campbell2022}. 

\subsection{Stretch morphisms}
 To illustrate the problem of determining a suitable analogue of Theorem 
 \ref{theoremKnuth1973} according to our lifting $\mathscr{S}$ of West's stack-sorting map, 
 let us consider $\mathscr{S}$ evaluated at a permuting diagram that is $1$-stack-sortable. 

\begin{example}\label{trianglealign}
 We see that the permutation $312$ avoids the pattern $231$. So, let us consider $\mathscr{S}$ evaluated
 at the permutation diagram corresponding to $312$, as below: 
\begin{align}
 & \mathscr{S}\left( \begin{tikzpicture}[scale = 0.5,thick, baseline={(0,-1ex/2)}] 
\tikzstyle{vertex} = [shape = circle, minimum size = 7pt, inner sep = 1pt] 
\node[vertex] (G--3) at (3.0, -1) [shape = circle, draw] {}; 
\node[vertex] (G-1) at (0.0, 1) [shape = circle, draw] {}; 
\node[vertex] (G--2) at (1.5, -1) [shape = circle, draw] {}; 
\node[vertex] (G-3) at (3.0, 1) [shape = circle, draw] {}; 
\node[vertex] (G--1) at (0.0, -1) [shape = circle, draw] {}; 
\node[vertex] (G-2) at (1.5, 1) [shape = circle, draw] {}; 
\draw[] (G-1) .. controls +(1, -1) and +(-1, 1) .. (G--3); 
\draw[] (G-3) .. controls +(-0.75, -1) and +(0.75, 1) .. (G--2); 
\draw[] (G-2) .. controls +(-0.75, -1) and +(0.75, 1) .. (G--1); 
\end{tikzpicture} \right) = \label{firstlabeltriangle} \\ 
 & \mathscr{S}\left( \begin{tikzpicture}[scale = 0.5,thick, baseline={(0,-1ex/2)}] 
\tikzstyle{vertex} = [shape = circle, minimum size = 7pt, inner sep = 1pt] 
\node[vertex] (G--3) at (3.0, -1) [shape = circle, draw,fill=black] {}; 
\node[vertex] (G-1) at (0.0, 1) [shape = circle, draw,fill=black] {}; 
\node[vertex] (G--2) at (1.5, -1) [shape = circle, draw,fill=white] {}; 
\node[vertex] (G-3) at (3.0, 1) [shape = circle, draw,fill=white] {}; 
\node[vertex] (G--1) at (0.0, -1) [shape = circle, draw,fill=white] {}; 
\node[vertex] (G-2) at (1.5, 1) [shape = circle, draw,fill=white] {}; 
\draw[] (G-3) .. controls +(-0.75, -1) and +(0.75, 1) .. (G--2); 
\draw[] (G-2) .. controls +(-0.75, -1) and +(0.75, 1) .. (G--1); 
\end{tikzpicture} \right) \left( \begin{tikzpicture}[scale = 0.5,thick, baseline={(0,-1ex/2)}] 
\tikzstyle{vertex} = [shape = circle, minimum size = 7pt, inner sep = 1pt] 
\node[vertex] (G--3) at (3.0, -1) [shape = circle, draw] {}; 
\node[vertex] (G-1) at (0.0, 1) [shape = circle, draw] {}; 
\node[vertex] (G--2) at (1.5, -1) [shape = circle, draw,fill=black] {}; 
\node[vertex] (G-3) at (3.0, 1) [shape = circle, draw,fill=black] {}; 
\node[vertex] (G--1) at (0.0, -1) [shape = circle, draw,fill=black] {}; 
\node[vertex] (G-2) at (1.5, 1) [shape = circle, draw,fill=black] {}; 
\draw[] (G-1) .. controls +(1, -1) and +(-1, 1) .. (G--3); 
\end{tikzpicture} \right) = \nonumber \\ 
 & \mathscr{S}\left( \begin{tikzpicture}[scale = 0.5,thick, baseline={(0,-1ex/2)}] 
\tikzstyle{vertex} = [shape = circle, minimum size = 7pt, inner sep = 1pt] 
\node[vertex] (G--3) at (3.0, -1) [shape = circle, draw,fill=black] {}; 
\node[vertex] (G-1) at (0.0, 1) [shape = circle, draw,fill=black] {}; 
\node[vertex] (G--2) at (1.5, -1) [shape = circle, draw,fill=black] {}; 
\node[vertex] (G-3) at (3.0, 1) [shape = circle, draw,fill=black] {}; 
\node[vertex] (G--1) at (0.0, -1) [shape = circle, draw,fill=white] {}; 
\node[vertex] (G-2) at (1.5, 1) [shape = circle, draw,fill=white] {}; 
\draw[] (G-2) .. controls +(-0.75, -1) and +(0.75, 1) .. (G--1); 
\end{tikzpicture} \right) 
 \left( \begin{tikzpicture}[scale = 0.5,thick, baseline={(0,-1ex/2)}] 
\tikzstyle{vertex} = [shape = circle, minimum size = 7pt, inner sep = 1pt] 
\node[vertex] (G--3) at (3.0, -1) [shape = circle, draw,fill=black] {}; 
\node[vertex] (G-1) at (0.0, 1) [shape = circle, draw,fill=black] {}; 
\node[vertex] (G--2) at (1.5, -1) [shape = circle, draw,fill=white] {}; 
\node[vertex] (G-3) at (3.0, 1) [shape = circle, draw,fill=white] {}; 
\node[vertex] (G--1) at (0.0, -1) [shape = circle, draw,fill=black] {}; 
\node[vertex] (G-2) at (1.5, 1) [shape = circle, draw,fill=black] {}; 
\draw[] (G-3) .. controls +(-0.75, -1) and +(0.75, 1) .. (G--2); 
\end{tikzpicture} \right) \left( \begin{tikzpicture}[scale = 0.5,thick, baseline={(0,-1ex/2)}] 
\tikzstyle{vertex} = [shape = circle, minimum size = 7pt, inner sep = 1pt] 
\node[vertex] (G--3) at (3.0, -1) [shape = circle, draw] {}; 
\node[vertex] (G-1) at (0.0, 1) [shape = circle, draw] {}; 
\node[vertex] (G--2) at (1.5, -1) [shape = circle, draw,fill=black] {}; 
\node[vertex] (G-3) at (3.0, 1) [shape = circle, draw,fill=black] {}; 
\node[vertex] (G--1) at (0.0, -1) [shape = circle, draw,fill=black] {}; 
\node[vertex] (G-2) at (1.5, 1) [shape = circle, draw,fill=black] {}; 
\draw[] (G-1) .. controls +(1, -1) and +(-1, 1) .. (G--3); 
\end{tikzpicture} \right). \nonumber 
\end{align}
 Following through with the procedure introduced in Section \ref{sectionlift}, 
 we obtain that identity permutation diagram in $\mathscr{P}_{3}$ shown below. 
 $$ 
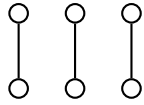
\begin{tikzpicture}[scale = 0.5,thick, baseline={(0,-1ex/2)}] 
\tikzstyle{vertex} = [shape = circle, minimum size = 7pt, inner sep = 1pt] 
\node[vertex] (G--3) at (3.0, -1) [shape = circle, draw] {}; 
\node[vertex] (G-3) at (3.0, 1) [shape = circle, draw] {}; 
\node[vertex] (G--2) at (1.5, -1) [shape = circle, draw] {}; 
\node[vertex] (G-2) at (1.5, 1) [shape = circle, draw] {}; 
\node[vertex] (G--1) at (0.0, -1) [shape = circle, draw] {}; 
\node[vertex] (G-1) at (0.0, 1) [shape = circle, draw] {}; 
\draw[] (G-3) .. controls +(0, -1) and +(0, 1) .. (G--3); 
\draw[] (G-2) .. controls +(0, -1) and +(0, 1) .. (G--2); 
\draw[] (G-1) .. controls +(0, -1) and +(0, 1) .. (G--1); 
\end{tikzpicture} $$ 
\end{example}

 From Theorem \ref{theoremlift} together with Example \ref{trianglealign}, given a partition diagram $\pi \in \mathscr{P}_{n}$, to generalize the notion 
 of $1$-stack-sortability in an applicable and meaningful way, it would be appropriate to use a condition whereby $\mathscr{S}(\pi)$ belongs to a fixed class of 
 generalizations of the multiplicative identity element 
\begin{equation}\label{arbitraryid}
 \underbrace{ \begin{tikzpicture}[scale = 0.5,thick, baseline={(0,-1ex/2)}] 
\tikzstyle{vertex} = [shape = circle, minimum size = 7pt, inner sep = 1pt] 
\node[vertex] (G--2) at (1.5, -1) [shape = circle, draw] {}; 
\node[vertex] (G-2) at (1.5, 1) [shape = circle, draw] {}; 
\node[vertex] (G--1) at (0.0, -1) [shape = circle, draw] {}; 
\node[vertex] (G-1) at (0.0, 1) [shape = circle, draw] {}; 
\draw[] (G-2) .. controls +(0, -1) and +(0, 1) .. (G--2); 
\draw[] (G-1) .. controls +(0, -1) and +(0, 1) .. (G--1); 
\end{tikzpicture}
 \cdots \begin{tikzpicture}[scale = 0.5,thick, baseline={(0,-1ex/2)}] 
\tikzstyle{vertex} = [shape = circle, minimum size = 7pt, inner sep = 1pt] 
\node[vertex] (G--1) at (0.0, -1) [shape = circle, draw] {}; 
\node[vertex] (G-1) at (0.0, 1) [shape = circle, draw] {}; 
\draw[] (G-1) .. controls +(0, -1) and +(0, 1) .. (G--1); 
\end{tikzpicture}}_{n} 
\end{equation} 
 in $\mathscr{P}_{n}$. This leads us to apply, as below, a morphism on partition algebras 
 introduced in a representation-theoretic context in \cite{Campbell2022}. 

 Since our lifting of West's stack-sorting map applies to partition diagrams in general, as opposed to, say, rook diagrams, it would be appropriate to allow 
 non-rook diagrams in our lifting of the concept of $1$-stack sortability. However, our lifting of $s\colon S_{n} \to S_{n}$ is such that it preserves the sizes of 
 blocks in a partition diagram (but not necessarily the arrangement or ordering of the blocks). 
 So, it would be appropriate to let a partition diagram be mapped, under $\mathscr{S}$, 
 to a generalization of \eqref{arbitraryid} allowing for the possibility of blocks other than $2$-blocks. 
 In this regard, we are to employ what is referred to as the \emph{Stretch} morphism for partition algebras, 
 as introduced in \cite{Campbell2022}. 

 Let $S$ be a finite set of natural numbers. Let $\pi$ be a set-partition of $S \cup S'$.
 Let $k$ be a natural number such that $k \geq \max(S)$. 
 Following \cite{Campbell2022}, we write $\delta_{k}(\pi)$
 to denote the diagram basis element in $\mathscr{P}_{n}^{\xi}$
 corresponding to the set-partition given by adding blocks of the form $\{ i, i' \}$
 to $\pi$, where $i$ is a natural number such that $i \not\in S$ and $i \leq k$. 
 Again, following \cite{Campbell2022}, 
 we let 
 $\alpha = (\alpha_{1}, \alpha_{2}, \ldots, \alpha_{ \ell(\alpha) })$ be a set-composition of a finite set of natural numbers 
 (referring to \cite{Campbell2022} for details on combinatorial terms related to Stretch morphisms), 
 and we write 
 $m = \ell(\alpha)$, 
 and set $k \geq \max \left( \bigcup \alpha \right)$. As in \cite{Campbell2022}, we define 
\begin{equation}\label{Stretchfromto}
 \text{{Stretch}}_{\alpha, k}\colon \mathscr{P}_{m}^{\xi} \to \mathscr{P}_{k}^{\xi} 
\end{equation}
 in the following manner. 

 Let $d_{\pi}$ be a 
 member of the diagram basis of $\mathscr{P}_{m}^{\xi}$. Let us write 
 $\pi = \{ \pi_{1}, \pi_{2}, \ldots, \pi_{\ell(\pi)} \}$. Then 
\begin{equation}\label{Stretchdelta}
 \text{{Stretch}}_{\alpha, k}(d_{\pi}) 
 = \delta_k \Bigg(\Bigg\{ \ \raisebox{0.5ex}{$ \displaystyle 
 \bigcup_{\substack{ i \in \pi_j \\ \text{$i$ is unprimed} }} \alpha_{i} 
 \cup \bigcup_{i' \in \pi_j} \alpha_i' : 1 \leq j \leq \ell(\pi) $} \ \Bigg\}\Bigg). 
\end{equation}
 We may extend this definition linearly as in \cite{Campbell2022} 
 so as to obtain an algebra homormophism, but this is not important for our purposes. 

 An interesting feature concerning the problem of generalizing $1$-stack-sortability by generalizing \eqref{smap} so as to allow partition diagrams as the 
 arguments of   a  lifting/extension of  the  
       $s$-map may be explained as follows. By enforcing different conditions on how \eqref{arbitraryid} 
 should be generalized in order to generalize $1$-stack-sortability, say, in the context of a given application, 
 this gives rise to different families of combinatorial objects in the study of 
 the classification of partition diagrams that reduce to a specified generalization of 
 \eqref{arbitraryid}, under the application of $\mathscr{S}$. 

\begin{definition}\label{definitionsss}
 We say that a partition diagram $\pi$ is \emph{stretch-stack-sortable} 
 if $\mathscr{S}(\pi)$ is equal to a Stretch map 
 evaluated at an identity permutation diagram. 
\end{definition}

 To begin with, we illustrate the effect of applying \eqref{Stretchdelta} to the multiplicative identity elements in a partition algebra. 

\begin{example}
 With regard to the notation in \eqref{Stretchdelta}, we let $d_{\pi}$ be the partition diagram corresponding to 
\begin{equation}\label{id4}
 \pi = \{ \{ 1, 1' \}, \{ 2, 2' \}, \{ 3, 3' \}, \{ 4, 4' \} \}
\end{equation} 
 Define $\alpha$ as the set-composition $( \{ 1, 2 \}$, $ \{ 3 \}$, $ \{ 5, 6, 7 \}$, $ \{ 4 \} )$. 
 Observe that the length of $\alpha$ equals the order of $d_{\pi}$, in accordance with the definition in 
 \eqref{Stretchdelta}. 
 Let us consider the $j = 3$ case within the family of the argument of $\delta_{k}$ indicated in 
 \eqref{Stretchdelta}, letting the elements in $\pi$
 be ordered in the manner we have written these elements in \eqref{id4}. 
 So, the element corresponding to this $j = 3$ case is 
\begin{align*}
 & \bigcup_{\substack{ i \in \pi_3 \\ \text{$i$ is unprimed} }} \alpha_{i} \cup \bigcup_{i' \in \pi_3} \alpha_i' = 
 \alpha_{3} \cup \alpha_{3}' = 
 \{ 5, 6, 7, 5', 6', 7' \}. 
\end{align*}
 Continuing similarly with respect to the other possible values for the $j$-index, 
 we obtain that $\text{Stretch}_{\alpha, 7}$ evaluated at $d_{\pi}$ is 
 equal to the following. 
$$ 
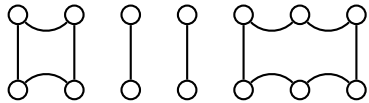
\begin{tikzpicture}[scale = 0.5,thick, baseline={(0,-1ex/2)}] 
\tikzstyle{vertex} = [shape = circle, minimum size = 7pt, inner sep = 1pt] 
\node[vertex] (G--7) at (9.0, -1) [shape = circle, draw] {}; 
\node[vertex] (G--6) at (7.5, -1) [shape = circle, draw] {}; 
\node[vertex] (G--5) at (6.0, -1) [shape = circle, draw] {}; 
\node[vertex] (G-5) at (6.0, 1) [shape = circle, draw] {}; 
\node[vertex] (G-6) at (7.5, 1) [shape = circle, draw] {}; 
\node[vertex] (G-7) at (9.0, 1) [shape = circle, draw] {}; 
\node[vertex] (G--4) at (4.5, -1) [shape = circle, draw] {}; 
\node[vertex] (G-4) at (4.5, 1) [shape = circle, draw] {}; 
\node[vertex] (G--3) at (3.0, -1) [shape = circle, draw] {}; 
\node[vertex] (G-3) at (3.0, 1) [shape = circle, draw] {}; 
\node[vertex] (G--2) at (1.5, -1) [shape = circle, draw] {}; 
\node[vertex] (G--1) at (0.0, -1) [shape = circle, draw] {}; 
\node[vertex] (G-1) at (0.0, 1) [shape = circle, draw] {}; 
\node[vertex] (G-2) at (1.5, 1) [shape = circle, draw] {}; 
\draw[] (G-5) .. controls +(0.5, -0.5) and +(-0.5, -0.5) .. (G-6); 
\draw[] (G-6) .. controls +(0.5, -0.5) and +(-0.5, -0.5) .. (G-7); 
\draw[] (G-7) .. controls +(0, -1) and +(0, 1) .. (G--7); 
\draw[] (G--7) .. controls +(-0.5, 0.5) and +(0.5, 0.5) .. (G--6); 
\draw[] (G--6) .. controls +(-0.5, 0.5) and +(0.5, 0.5) .. (G--5); 
\draw[] (G--5) .. controls +(0, 1) and +(0, -1) .. (G-5); 
\draw[] (G-4) .. controls +(0, -1) and +(0, 1) .. (G--4); 
\draw[] (G-3) .. controls +(0, -1) and +(0, 1) .. (G--3); 
\draw[] (G-1) .. controls +(0.5, -0.5) and +(-0.5, -0.5) .. (G-2); 
\draw[] (G-2) .. controls +(0, -1) and +(0, 1) .. (G--2); 
\draw[] (G--2) .. controls +(-0.5, 0.5) and +(0.5, 0.5) .. (G--1); 
\draw[] (G--1) .. controls +(0, 1) and +(0, -1) .. (G-1); 
\end{tikzpicture} $$
\end{example}

 Now, let us consider an illustration of a partition diagram satisfying the conditions of 
 Definition \ref{definitionsss}. 

\begin{example}\label{exwith112}
 According to the procedure in Section \ref{sectionlift}, we may verify the evaluation
\begin{equation}\label{20221208415TABTABTYABAM1A}
 \mathscr{S}\left( \begin{tikzpicture}[scale = 0.5,thick, baseline={(0,-1ex/2)}] 
\tikzstyle{vertex} = [shape = circle, minimum size = 7pt, inner sep = 1pt] 
\node[vertex] (G--4) at (4.5, -1) [shape = circle, draw] {}; 
\node[vertex] (G-1) at (0.0, 1) [shape = circle, draw] {}; 
\node[vertex] (G--3) at (3.0, -1) [shape = circle, draw] {}; 
\node[vertex] (G--2) at (1.5, -1) [shape = circle, draw] {}; 
\node[vertex] (G-3) at (3.0, 1) [shape = circle, draw] {}; 
\node[vertex] (G-4) at (4.5, 1) [shape = circle, draw] {}; 
\node[vertex] (G--1) at (0.0, -1) [shape = circle, draw] {}; 
\node[vertex] (G-2) at (1.5, 1) [shape = circle, draw] {}; 
\draw[] (G-1) .. controls +(1, -1) and +(-1, 1) .. (G--4); 
\draw[] (G-3) .. controls +(0.5, -0.5) and +(-0.5, -0.5) .. (G-4); 
\draw[] (G-4) .. controls +(-0.75, -1) and +(0.75, 1) .. (G--3); 
\draw[] (G--3) .. controls +(-0.5, 0.5) and +(0.5, 0.5) .. (G--2); 
\draw[] (G--2) .. controls +(0.75, 1) and +(-0.75, -1) .. (G-3); 
\draw[] (G-2) .. controls +(-0.75, -1) and +(0.75, 1) .. (G--1); 
\end{tikzpicture} \right) 
 = \begin{tikzpicture}[scale = 0.5,thick, baseline={(0,-1ex/2)}] 
\tikzstyle{vertex} = [shape = circle, minimum size = 7pt, inner sep = 1pt] 
\node[vertex] (G--4) at (4.5, -1) [shape = circle, draw] {}; 
\node[vertex] (G-4) at (4.5, 1) [shape = circle, draw] {}; 
\node[vertex] (G--3) at (3.0, -1) [shape = circle, draw] {}; 
\node[vertex] (G--2) at (1.5, -1) [shape = circle, draw] {}; 
\node[vertex] (G-2) at (1.5, 1) [shape = circle, draw] {}; 
\node[vertex] (G-3) at (3.0, 1) [shape = circle, draw] {}; 
\node[vertex] (G--1) at (0.0, -1) [shape = circle, draw] {}; 
\node[vertex] (G-1) at (0.0, 1) [shape = circle, draw] {}; 
\draw[] (G-4) .. controls +(0, -1) and +(0, 1) .. (G--4); 
\draw[] (G-2) .. controls +(0.5, -0.5) and +(-0.5, -0.5) .. (G-3); 
\draw[] (G-3) .. controls +(0, -1) and +(0, 1) .. (G--3); 
\draw[] (G--3) .. controls +(-0.5, 0.5) and +(0.5, 0.5) .. (G--2); 
\draw[] (G--2) .. controls +(0, 1) and +(0, -1) .. (G-2); 
\draw[] (G-1) .. controls +(0, -1) and +(0, 1) .. (G--1); 
\end{tikzpicture}. 
\end{equation}
 So, the argument of $\mathscr{S}$, in this case, is stretch-stack-sortable. 
\end{example}

\begin{example}\label{skew}
 We also find that 
 $$
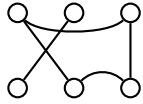
\begin{tikzpicture}[scale = 0.5,thick, baseline={(0,-1ex/2)}] 
\tikzstyle{vertex} = [shape = circle, minimum size = 7pt, inner sep = 1pt] 
\node[vertex] (G--3) at (3.0, -1) [shape = circle, draw] {}; 
\node[vertex] (G--2) at (1.5, -1) [shape = circle, draw] {}; 
\node[vertex] (G-1) at (0.0, 1) [shape = circle, draw] {}; 
\node[vertex] (G-3) at (3.0, 1) [shape = circle, draw] {}; 
\node[vertex] (G--1) at (0.0, -1) [shape = circle, draw] {}; 
\node[vertex] (G-2) at (1.5, 1) [shape = circle, draw] {}; 
\draw[] (G-1) .. controls +(0.6, -0.6) and +(-0.6, -0.6) .. (G-3); 
\draw[] (G-3) .. controls +(0, -1) and +(0, 1) .. (G--3); 
\draw[] (G--3) .. controls +(-0.5, 0.5) and +(0.5, 0.5) .. (G--2); 
\draw[] (G--2) .. controls +(-0.75, 1) and +(0.75, -1) .. (G-1); 
\draw[] (G-2) .. controls +(-0.75, -1) and +(0.75, 1) .. (G--1); 
\end{tikzpicture}$$
 it stretch-stack-sortable. 
\end{example}

 Examples \ref{exwith112} and \ref{skew} illustrate the problem of classifying stretch-stack-sortable partition diagrams. 
 The importance of diagram algebras in both the study of algebraic groups and in the field of knot theory provides 
 a source of further motivation concerning out interest in lifting and generalizing 
 the notions of stack-sortability and pattern avoidance in permutations to basis elements in diagram algebras. 
 The following Theorem appears to be the first direct step in this direction. 
  The   situation indicated via the  bullet points    given in  Theorem   \ref{maintheorem} 
  formalizes   how  the  notion  of  avoiding  the pattern  $231$ 
  may be lifted from permutations to partition diagrams, 
  in  a way that is directly applicable to  
  lifting  West's stack-sorting map. 

\begin{theorem}\label{maintheorem}
 A partition diagram $\pi$ is stretch-stack-sortable if and only if: 

\begin{enumerate}

\item The only blocks of $\pi$ are propagating; 

\item For each block of $\pi$, it has the same number of upper and lower vertices; 

\item For each block of $\pi$, all of its lower nodes are consecutive as primed integers; and 

\item Letting      the blocks of $\pi$ be denoted in the form $C_{i} \cup D_{i}'$ for indices $i$, 
   and where $C_{i}$ and $D_{i}'$ respectively denote the set of upper and lower   
    nodes of a given block, and    writing    
\begin{equation}\label{Dprimechain}
 D_{1}' < D_{2}' < \cdots, 
\end{equation}
 according to the consecutive primed integers labeling $D_{i}'$, 
  the         situation indicated via the following bullet points cannot occur.    

$\bullet$ $D_{i_{1}}' < D_{i_{2}}' < D_{i_{3}}'$; 

$\bullet$ By applying the procedure used to define $\mathscr{S}$, 
 at the stage in this application when $D_{i_{3}}'$ contains the greatest non-singleton primed nodes, 
   one of the following situations occurs: 

$\bullet$ Either $D_{i_{2}}'$ is part of an $\mathscr{L}$-configuration and $D_{i_{1}}'$ is part of an $\mathscr{R}$-configuration; or 

$\bullet$ $D_{i_{2}}'$ is part of an $\mathscr{L}$-configuration and $D_{i_{1}}'$ is part of an $\mathscr{M}$-configuration; or 

$\bullet$ $D_{i_{2}}'$ is part of an $\mathscr{M}$-configuration and $D_{i_{1}}'$ is part of an $\mathscr{R}$-configuration; or 

$\bullet$ $D_{i_{2}}'$ is part of an $\mathscr{M}_{j}$-configuration and $D_{i_{1}}'$ is part of an $\mathscr{M}_{k}$-configuration, 
 with $j < k$. 

\end{enumerate}

\end{theorem}

\begin{proof}
 $(\Longrightarrow)$ Let $\pi$ be a stretch-stack-sortable partition diagram. The mapping $\mathscr{S}$ is such that the number of blocks in a partition 
 diagram $\mu$ with $u$ upper nodes and $l$ lower nodes is equal to the number of blocks in $\mathscr{S}(\mu)$
    with $u$ upper nodes and $l$ lower nodes. 
 So, $\pi$ cannot have any non-propagating blocks, because, otherwise, $\mathscr{\pi}$ would have a non-propagating block, contradicting that $\pi$ is 
 stretch-stack-sortable. By using the same property of $\mathscr{S}$ that we have previously indicated, we have that each block of $\pi$ is such that 
 it has the same number of upper and lower vertices. 
 Let $B$ be a block of $\pi$, and, by way of contradiction, 
 suppose that it is not the case that all of its lower nodes are consecutive as primed integers. 
 According to the procedure used to define $\mathscr{S}$, 
 in $\mathscr{S}(\pi)$, there would have to be a new block $N$ 
 such that the labels for the lower nodes of $N$ 
 are the same as the labels of the lower nodes of $B$
 and such that the upper nodes of $N$ are consecutive integers.
 So, there would be a block $N$ in $\mathscr{S}(\pi)$ 
 with consecutive integers labeling the upper nodes of $N$
 and such that the lower nodes of $N$ do not form a set of consecutive prime integers. 
 However, this is impossible for the stretch of an identity diagram. 
 Now, by way of contradiction, suppose that the situation indicated via the above six bullet points occurs. 
 In any out of the four possibilities corresponding to the last four bullet points, the removal 
 of the block containing $D_{i_{3}}'$ has the effect of separating 
 the partition diagram containing $D_{i_{1}}'$ and $D_{i_{2}}'$ 
 in such a way so that in the $\odot$-product shown in \eqref{maindecomposition}, 
 the block containing $D_{i_{1}}'$ will be in an argument of $\mathscr{S}$ strictly to 
 the left of the factor in the $\odot$-product given by $\mathscr{S}$ evaluated 
 at an expression involving $D_{i_{2}}'$. Consequently, 
 the ordering of the positions of the bottom nodes of $D_{i_{1}}'$ and $D_{i_{2}}'$ will be reversed, 
 but then the top nodes corresponding to $D_{i_{2}}'$ will be labeled with integers 
 strictly smaller than the top nodes corresponding to $D_{i_{1}}'$ 
 in the evaluation of $\mathscr{S}(\pi)$, giving us a crossing in the partition diagram 
 $\mathscr{S}(\pi)$ that would be impossible in the stretch of an identity partition diagram. 

 \ 

\noindent $(\Longleftarrow)$ Conversely, suppose that a partition diagram $\pi \in \mathscr{P}_{n}$ satisfies all four of the conditions listed above. 
 We apply $\mathscr{S}$ to $\pi$, so as to obtain a decomposition of the form indicated in \eqref{maindecomposition}. Adopting notation from 
 \eqref{maindecomposition}, we may deduce that the following properties hold, again working under the assumption that the four conditions in the 
 Theorem under consideration hold: (i) The expression $B'$ consists, apart from singleton blocks, of a single propagating block with the same number of 
 upper and lower nodes, 
 and such that the lower nodes are labeled with consecutive primed integers ending  with $n'$; 
 and (ii) If we were to keep the labelings for the lower nodes as in $\pi'$, 
 then 
 the labeling for the lower non-singleton blocks in the $\odot$-decomposition in 
 \eqref{maindecomposition} would be in the same order. 
 This latter property may be verified by a case analysis 
 of the form suggested by the last four out of the six bullet points given above. 
 Repeating this argument inductively, 
 and mimicking notation from \eqref{unorderedodot}, 
 we find that $\mathscr{S}(\pi)$ 
 may be written as an $\odot$-product of the form 
\begin{equation}\label{scrBodot}
 \mathscr{B}_{1}' \odot \mathscr{B}_{2}' \odot \cdots \odot \mathscr{B}_{m_{2}}', 
\end{equation}
 where each expression of the form $B_{i}'$ consists of, apart from singleton blocks, only one propagating block, and this propagating block has the same 
 number of top nodes as bottom nodes and has consecutive primed integers as its bottom nodes; moreover, the bottom nodes according to the ordering in 
 \eqref{scrBodot} form the sequence $1' < 2' < \cdots < n'$. So, a direct application of the labeling according to the definition for the $\odot$-product gives 
 us a stretch of an identity partition diagram. 
\end{proof}

\begin{example}
 If we return to the partition diagram in $\mathscr{P}_{9}$ 
 that is illustrated in \eqref{clearillustrationmain}
 within Example \ref{exampleambiguity}, we find that there is a propagating block containing the largest
 bottom node $9'$. 
 Being consistent with our notation in \eqref{Dprimechain}, 
 we write $D_{1}' = \{ 1', 2', 3' \}$, $D_{2}' = \{ 4', 5', 6' \}$, and $D_{3}' = \{ 7', 8', 9' \}$, 
 and we let the corresponding blocks in \eqref{clearillustrationmain}
 be denoted as $C_{1} \cup D_{1}'$, $C_{2} \cup D_{2}'$, and $C_{3} \cup D_{3}'$. 
 We have that $D_{1}' < D_{2}' < D_{3}'$, but when we   apply  the procedure used to define $\mathscr{S}$
 to the partition diagram in \eqref{clearillustrationmain}, 
 as we remove the block $C_{3} \cup D_{3}'$, we find that 
 $D_{2}'$ is part of an $\mathscr{L}$-configuration and $D_{1}'$ is part of an $\mathscr{M}$-configuration, 
 which is one of the forbidden patterns indicated in Theorem \ref{maintheorem}. 
 So, according to Theorem \ref{maintheorem}, the partition diagram in \eqref{clearillustrationmain}
 should not be stretch-stack-sortable, and we may verify that $\mathscr{S}$ evaluated at this same 
 diagram in \eqref{clearillustrationmain} is equal to 
\begin{equation}\label{72707272717271727170757A7M71A}
 \begin{tikzpicture}[scale = 0.5,thick, baseline={(0,-1ex/2)}] 
\tikzstyle{vertex} = [shape = circle, minimum size = 7pt, inner sep = 1pt] 
\node[vertex] (G--9) at (12.0, -1) [shape = circle, draw,fill=red] {}; 
\node[vertex] (G--8) at (10.5, -1) [shape = circle, draw,fill=red] {}; 
\node[vertex] (G--7) at (9.0, -1) [shape = circle, draw,fill=red] {}; 
\node[vertex] (G-7) at (9.0, 1) [shape = circle, draw,fill=red] {}; 
\node[vertex] (G-8) at (10.5, 1) [shape = circle, draw,fill=red] {}; 
\node[vertex] (G-9) at (12.0, 1) [shape = circle, draw,fill=red] {}; 
\node[vertex] (G--6) at (7.5, -1) [shape = circle, draw,fill=cyan] {}; 
\node[vertex] (G--5) at (6.0, -1) [shape = circle, draw,fill=cyan] {}; 
\node[vertex] (G--4) at (4.5, -1) [shape = circle, draw,fill=cyan] {}; 
\node[vertex] (G-1) at (0.0, 1) [shape = circle, draw,fill=cyan] {}; 
\node[vertex] (G-2) at (1.5, 1) [shape = circle, draw,fill=cyan] {}; 
\node[vertex] (G-3) at (3.0, 1) [shape = circle, draw,fill=cyan] {}; 
\node[vertex] (G--3) at (3.0, -1) [shape = circle, draw,fill=green] {}; 
\node[vertex] (G--2) at (1.5, -1) [shape = circle, draw,fill=green] {}; 
\node[vertex] (G--1) at (0.0, -1) [shape = circle, draw,fill=green] {}; 
\node[vertex] (G-4) at (4.5, 1) [shape = circle, draw,fill=green] {}; 
\node[vertex] (G-5) at (6.0, 1) [shape = circle, draw,fill=green] {}; 
\node[vertex] (G-6) at (7.5, 1) [shape = circle, draw,fill=green] {}; 
\draw[] (G-7) .. controls +(0.5, -0.5) and +(-0.5, -0.5) .. (G-8); 
\draw[] (G-8) .. controls +(0.5, -0.5) and +(-0.5, -0.5) .. (G-9); 
\draw[] (G-9) .. controls +(0, -1) and +(0, 1) .. (G--9); 
\draw[] (G--9) .. controls +(-0.5, 0.5) and +(0.5, 0.5) .. (G--8); 
\draw[] (G--8) .. controls +(-0.5, 0.5) and +(0.5, 0.5) .. (G--7); 
\draw[] (G--7) .. controls +(0, 1) and +(0, -1) .. (G-7); 
\draw[] (G-1) .. controls +(0.5, -0.5) and +(-0.5, -0.5) .. (G-2); 
\draw[] (G-2) .. controls +(0.5, -0.5) and +(-0.5, -0.5) .. (G-3); 
\draw[] (G-3) .. controls +(1, -1) and +(-1, 1) .. (G--6); 
\draw[] (G--6) .. controls +(-0.5, 0.5) and +(0.5, 0.5) .. (G--5); 
\draw[] (G--5) .. controls +(-0.5, 0.5) and +(0.5, 0.5) .. (G--4); 
\draw[] (G--4) .. controls +(-1, 1) and +(1, -1) .. (G-1); 
\draw[] (G-4) .. controls +(0.5, -0.5) and +(-0.5, -0.5) .. (G-5); 
\draw[] (G-5) .. controls +(0.5, -0.5) and +(-0.5, -0.5) .. (G-6); 
\draw[] (G-6) .. controls +(-1, -1) and +(1, 1) .. (G--3); 
\draw[] (G--3) .. controls +(-0.5, 0.5) and +(0.5, 0.5) .. (G--2); 
\draw[] (G--2) .. controls +(-0.5, 0.5) and +(0.5, 0.5) .. (G--1); 
\draw[] (G--1) .. controls +(1, 1) and +(-1, -1) .. (G-4); 
\end{tikzpicture},
\end{equation}
 being consistent with the colouring in \eqref{clearillustrationmain}. 
 We see that the partition diagram in \eqref{72707272717271727170757A7M71A} 
 is not a Stretch of any identity partition diagram. 
\end{example}

\section{Future work and applications}
 In regard to Knuth's famous result that the number of $1$-stack-sortable permutations 
 in $S_{n}$ is equal to the $n^{\text{th}}$ Catalan number $\frac{1}{n+1} \binom{2n}{n}$, 
 it would be desirable to obtain a meaningfully similar result 
 for counting the elements in $\mathscr{P}_{n}$
 satisfying the conditions in Theorem \ref{maintheorem}, 
 i.e., the number of stretch-stack-sortable 
 elements in $\mathscr{P}_{n}$.
 We leave this as an open problem. 

 An advantage of our lifting of West's stack-sorting map may be described as follows. Since our mapping $\mathscr{S}$ may be evaluated an an arbitrary 
 partition diagram, this allows us to apply and experiment with different ways of generalizing stack-sortability, based on different kinds of partition diagram 
 generalizations of identity permutations. We may obtain many further combinatorial results, relative to our  main results, through the use of variants and 
 generalizations of Definition \ref{definitionsss}. For example, what kinds of bijective and enumerative results can we obtain, relative to the Theorems 
 introduced in this article, if we instead consider partition diagrams that get mapped, under the application  of $\mathscr{S}$, to stretches of rook 
 diagrams with $2$-blocks of the form $\{ i, i' \}$? What if non-rectangular and non-singleton blocks are allowed? 

 Many remarkable results and applications have concerned preimages under West's stack-sorting map, as in the work of Defant et al.\ in 
 \cite{DefantEngenMiller2020}, for example. What kinds of applications and combinatorial results can we obtain concerning preimages under our 
 lifting of $s$? 

 Since our definition for stretch-stack-sortability is a lifting of $1$-stack-sortability, this raises the question as to how the notion of 
 $t$-stack-sortability may be generalized, according to our procedure from Section \ref{sectionlift}. 
 The study of $t$-stack-sortability is widely known to be much more difficult even for the $t = 2$ case, relative to the $t = 1$ case; see 
 \cite[\S1.2.3]{Vatter2005}, for example. In consideration as to Zeilberger's renowned proof \cite{Zeilberger1992} of West's conjecture that the number 
 of $2$-stack-sortable permutations in $S_{n}$ is 
 $$ \frac{2 (3n )!}{ (n+1)! (2n+1)!}, $$ 
 this inspires the determination of an analogue of the above indicated enumerative result, using a lifting of $2$-stack-sortability 
 via our mapping $\mathscr{S}$. 

\subsection*{Acknowledgements}
 The author wants to thank Mike Zabrocki for having described to the author (and introduced to the author)
 the partition algebra morphisms of the form indicated in \eqref{Stretchdelta} 
 and how such morphisms may be applied in the representation theory for partition algebras.

\end{document}